\documentclass[11pt,a4paper]{amsart}
\usepackage[all]{xy}
\SelectTips{cm}{}  \calclayout
\usepackage{amsfonts,amssymb,amscd,amsmath,enumerate,verbatim,calc}

\usepackage{amscd,amssymb,amsopn,amsmath,amsthm,graphics,amsfonts,enumerate,verbatim,calc}
\usepackage[dvips]{graphicx}
\usepackage[colorlinks=true,linkcolor=blue,citecolor=blue]{hyperref}
\input xy
\xyoption{all}

\newcommand{\lrt}{\longrightarrow}

\newcommand{\C}{\mathbf{C} }

\newcommand{\K}{\mathbf{K} }
\newcommand{\X}{\mathbf{X} }
\newcommand{\Y}{\mathbf{Y} }
\newcommand{\I}{\mathbf{I} }

\newcommand{\F}{\mathbf{F} }

\newcommand{\CO}{\mathcal{O}}

\newcommand{\Z}{\mathbb{Z} }
\newcommand{\Q}{\mathbb{Q} }
\newcommand{\CE}{\mathcal{E}}

\newcommand{\CC}{\mathcal{C} }

\newcommand{\CF}{\mathcal{F} }
\newcommand{\CG}{\mathcal{G} }

\newcommand{\CK}{\mathcal{K} }
\newcommand{\CP}{\mathcal{P} }

\newcommand{\CS}{\mathcal{S} }

\newcommand{\CX}{\mathcal{X} }
\newcommand{\CY}{\mathcal{Y} }

\newcommand{\Flat}{{\rm{Flat}}}

\newcommand{\CBCOTX}{{\mathbf{C}^{\rm{b}}({\rm Cot} X)}}

\newcommand{\CBX}{{\mathbf{C}^{\rm{b}}({\mathfrak{Qco}} X)}}

\newcommand{\CBAFX}{{\mathbf{C}^{\rm{b}}_{\rm{ac}}(\mathrm{FlatX})}}

\newcommand{\KFX}{{\mathbf{K}({\Flat} X)}}
\newcommand{\DFX}{{\mathbf{D}({\Flat} X)}}
\newcommand{\KCOFX}{{\K({\rm{Cof}} X)}}
\newcommand{\KPIFX}{{\K({\rm{Pinf}} X)}}

\newcommand{\KPFX}{{\mathbf{K}_{\rm{p}}({\Flat} X)}}

\newcommand{\CPFX}{{\mathbf{C}_{\rm{p}}({\Flat} X)}}

\newcommand{\im}{{\rm{Im}}}

\newcommand{\pd}{{\rm{pd}}}

\newcommand{\cd}{{\rm{cd}}}

\newcommand{\Hom}{{\rm{Hom}}}
\newcommand{\Ext}{{\rm{Ext}}}

\newtheorem{theorem}{Theorem}[section]
\newtheorem{corollary}[theorem]{Corollary}
\newtheorem{lemma}[theorem]{Lemma}
\newtheorem{proposition}[theorem]{Proposition}

\theoremstyle{definition}
\newtheorem{definition}[theorem]{Definition}
\newtheorem{example}[theorem]{Example}

\newtheorem{remark}[theorem]{Remark}

\theoremstyle{plain}
\newtheorem{stheorem}{Theorem}[subsection]

\newtheorem{sproposition}[stheorem]{Proposition}

\theoremstyle{definition}

\numberwithin{equation}{section}
\begin{document}

\title[flat quasi-coherent sheaves of finite cotorsion dimension]
{flat quasi-coherent sheaves of finite cotorsion dimension}
\author[Esmaeil Hosseini]{Esmaeil Hosseini}

\address{Department of Mathematics, Shahid Chamran University of Ahvaz, Ahvaz, Iran}
 \email{esmaeilmath@gmail.com}

\begin{abstract}
Let $X$ be e quasi-compact and semi-separated scheme. If every flat
quasi-coherent sheaf has finite cotorsion dimension, we prove that
$X$ is $n$-perfect for some $n\geq 0$. If $X$ is coherent and
$n$-perfect(not necessarily of finite krull dimension), we prove
that every flat quasi-coherent sheaf has finite
 pure injective dimension. Also, we show that there is an
 equivalence $\KPIFX\lrt\DFX$ of homotopy categories, whenever
 $\KPIFX$ is the homotopy category of pure injective flat
 quasi-coherent sheaves and $\DFX$ is the pure derived category of
 flat quasi-coherent sheaves.
\end{abstract}
\maketitle
\section{Introduction}
In this paper, $X$ denotes a quasi-compact and semi-separated
scheme, $\CO_X$-modules are quasi-coherent sheaves on $X$ and all
rings are commutative with identity.

Let $\mathfrak{Qco}X$ be the category of $\CO_X$-modules. An
$\CO_X$-module $\CF$ is called flat if for each $p\in X$,
 $\CF_p$
is a flat $\CO_{X,p}$-module or equivalently the the functor
$\CF\otimes_{\CO_X}-:\mathfrak{Qco}X\lrt\mathfrak{Qco}X$ is exact.
 In \cite[4.2]{EE}, the authors  proved that
the pair $(\mathrm{FlatX},\mathrm{CotX})$ is  a complete cotorsion
theory in $\mathfrak{Qco}X$, whenever FlatX is the class of all flat
$\CO_X$-modules and
$\mathrm{CotX}$=$\{\CG\in\mathfrak{Qco}X:\forall\CF\in\mathrm{FlatX},
~~\Ext^1_X(\CF,\CG)=0\}$, is the class of all cotorsion
$\CO_X$-modules. So,
 for a given $\CO_X$-module $\CG$, we can define cd$\CG$(the cotorsion
dimension of $\CG$).

In this paper we study  those schemes $X$ such that over them every
flat $\CO_X$-module has finite cotorsion dimension. If $X$ is
affine, then every flat $\CO_X$-module has finite projective
dimension and so there exist an integer $n\geq 0$ such that for each
flat $\CO_X$-module $\CF$, $\pd\CF\leq n$(the projective dimension
of $\CF$). Unfortunately, this argument is not true when $X$ is
non-affine, since there is no non-zero projective $\CO_X$-modules.
In the following result we make a proof to such case.
\begin{theorem}
The following conditions are equivalent:\\
(i) Every flat $\CO_X$-module has finite cotorsion dimension.\\
(ii) $X$ is $n$-perfect for some integer $n\geq 0$.
\end{theorem}

In the remainder of this paper we show that every flat
$\CO_X$-module has finite pure injective dimension. As an
application, we prove that if $X$ is coherent $n$-perfect then there
is an equivalence $\KPIFX\lrt\KPFX$ of homotopy categories, whenever
$\KFX$ be the homotopy category of  complexes of flat
$\CO_X$-modules and $\KPIFX$ be the essential image of the homotopy
category of complexes of pure injective flat $\CO_X$-modules in
$\KFX$ in the sense of \cite{HS}.

\vspace{1cm}

\textbf{Setup}. In this paper,
$\mathfrak{U}=\{\mathrm{Spec}A_i\}_{i=1}^m$ denotes a
semi-separating cover of $X$(i.e.  each intersection of elements of
$\mathfrak{U}$ is also affine).
\section{Cotorsion envelope of bounded complexes}
Let $\CBX$ be the category of bounded complexes of $\CO_X$-modules,
$\CBAFX$ be the category of bounded acyclic complexes of flat
$\CO_X$-modules and $\CBCOTX$ be the category of bounded complexes
of cotorsion $\CO_X$-modules. In this section, we prove that the
pair $(\CBAFX,\CBCOTX)$ is a complete cotorsion theory in $\CBX$,
i.e. $\CBAFX= {}^\perp\CBCOTX$, $\CBAFX^\perp=\CBCOTX$  and it has
enough projectives. For notations and definitions see \cite{HS} and \cite{EJ}.
\begin{lemma}\label{107}
Let $\xymatrix@C-0.9pc{\mathbf{G}:
0\ar[r]&\CG'\ar[r]&\CG\ar[r]&\CG''\ar[r]&0}$ be an exact sequence of
$\CO_X$-modules. Then there exists a morphism $\phi:\F\lrt
\mathbf{G}$ of complexes whenever  $\F$ is a short exact  sequence
of flat $\CO_X$-modules.
\end{lemma}
\begin{proof}
Let
$\xymatrix@C-0.7pc@R-0.9pc{0\ar[r]&\CC''\ar[r]&\CF''\ar[r]&\CG''\ar[r]&0}$
be the flat cover of $\CG''$. Consider the  pullback diagram
\[\xymatrix@C-0.7pc@R-0.9pc{&&0\ar[d]&0\ar[d]\\&&\CC''\ar[d]\ar@{=}[r]&\CC''\ar[d]\\0\ar[r]&\CG'
\ar[r]\ar@{=}[d]&\CP\ar[r]\ar[d]&\CF''\ar[r]\ar[d]&0\\
0\ar[r]&\CG'\ar[r]^i&\CG\ar[r]\ar[d]&\CG''\ar[r]\ar[d]&0,\\
&&0&0}\] and let
$\xymatrix@C-0.7pc@R-0.9pc{0\ar[r]&\CC'\ar[r]&\CF\ar[r]^p&\CP\ar[r]&0}$
be the flat cover  of $\CP$.  Then the  pullback of $i$ and $p$
 completes the proof.
\end{proof}
Recall that,  a bounded complex $\F$ is called \emph{flat} if
$\F\in\CBAFX$ and  a bounded complex $\mathbf{C}$ is called
\emph{cotorsion} if  $\C\in\CBAFX^{\perp}$, where the orthogonal is
taken in the exact category
$\mathbf{C}^{\mathrm{b}}(\mathfrak{Qco}X)$. By similar argument that
used in \cite[Proposition 2.1]{HS}, we deduce the following
proposition.
\begin{proposition}\label{106}
Let $\C$ be a bounded complex. Then  $\C$ is cotorsion  if and only
if it is a complex of cotorsion $\CO_X$-modules.
\end{proposition}

\begin{theorem}\label{211}
Let $\X$ be a bounded complex of $\CO_X$-modules. Then there exists
an exact sequence
$\xymatrix@C-0.7pc@R-0.9pc{0\ar[r]&\C\ar[r]&\F\ar[r]&\X\ar[r]&0}$ of
complexes, where $\F$ is flat and $\C$ is cotorsion.
\end{theorem}
\begin{proof}
By \cite{Sp}, there exists  a quasi-isomorphism $f: \X[-1]\lrt\I$,
with $\I$ is a bounded complex  of injective $\CO_X$-modules. By
Lemma \ref{107}, we construct the short  exact sequence
$\xymatrix@C-0.7pc@R-0.9pc{0\ar[r]&\C'\ar[r]&\F\ar[r]&\mathrm{cone}(f)\ar[r]&0},$
with  $\F$ is flat and  $\C'$ is cotorsion complex. Then the
pullback of the morphisms
$\xymatrix@C-0.7pc{\F\ar[r]&\mathrm{cone}(f)}$ and
$\xymatrix@C-0.7pc{\I\ar[r]&\mathrm{cone}(f)}$ completes the proof.
\end{proof}
\begin{theorem}\label{108}
The pair $(\CBAFX,\CBCOTX)$ is a complete cotorsion theory in
$\CBX$.
\end{theorem}
\begin{proof}
It suffices to show that $\CBAFX={}^\perp\CBCOTX$.  Let
$\X\in{}^\perp\CBCOTX$. By Theorem \ref{211}, there exist an  exact
sequence
$\xymatrix@C-0.7pc@R-0.9pc{0\ar[r]&\C\ar[r]&\F\ar[r]&\X\ar[r]&0},$
of complexes, with $\F$ is flat and  $\C$ is cotorsion. By
assumption this is  split exact sequence. Then $\X\in\CBAFX$.
Therefore $(\CBAFX,\CBCOTX)$ is a cotorsion theory which is complete
by Theorem \ref{211}.
\end{proof}
\begin{corollary}
Every bounded complex of $\CO_X$-modules admits flat cover and
cotorsion envelope.
\end{corollary}
\section{cotorsion dimension of flat $\CO_X$-module}
In this section we prove that if every flat $\CO_X$-module has
finite cotorsion dimension then $X$ is  $n$-perfect for some $n\geq
0$.
\begin{definition}
Let $n\geq 0$ be an integer. $X$ is called $n$-perfect  if
$n=\mathrm{sup}\{\mathrm{cd}\CF| \CF\in\mathrm{Flat}X\}$.

\end{definition}

\begin{theorem}\label{nper}
Let for each $1\leq i\leq m$, every flat $A_i$-module has finite
cotorsion dimension.  Then $X$ is $n$-perfect for some $n$.
\end{theorem}
\begin{proof}
Let $\CF$ be a flat $\CO_X$-module and
\[\xymatrix@C-0.7pc@R-0.9pc{\mathbf{G}: 0\ar[r]&\CF\ar[r]&\mathfrak{C}^0(\mathfrak{U},\CF)\ar[r]
&\mathfrak{C}^1(\mathfrak{U},
\CF)\ar[r]&\cdots\ar[r]&\mathfrak{C}^{m-2}(\mathfrak{U},\CF)\ar[r]&\mathfrak{C}^{m-1}(\mathfrak{U},\CF)\ar[r]&
0}\] be its $\check{\mathrm{C}}$heck resolution. By assumption, for
each $0\leq i\leq m-1$, $\mathfrak{C}^i(\mathfrak{U},\CF)$ has
finite cotorsion dimension, then by Theorem \ref{108} there exist a
resolution
$$\xymatrix@C-0.7pc@R-0.9pc{
0\ar[r]&\mathbf{G}\ar[r]&\mathbf{C}_0\ar[r]&\mathbf{C}_1\ar[r]&\cdots\ar[r]&\mathbf{C}_{n-1}\ar[r]&\mathbf{C}_n\ar[r]&
0}$$  of $\mathbf{G}$ by  flat complexes of $\CO_X$-modules, where
$\mathbf{C}_0$, $\mathbf{C}_1$, ..., $\mathbf{C}_{n-1}$ are
cotorsion complexes and $\mathbf{C}_n$ is a flat complex  such that
for each $i> 0$, $\mathbf{C}_n^i$ is cotorsion. Then the flat
complex
\[\xymatrix@C-0.7pc@R-0.9pc{ 0\ar[r]&\CF\ar[r]&\mathbf{C}_0^0\ar[r]&\mathbf{C}_1^0\ar[r]&\cdots\ar[r]&\mathbf{C}_{n-1}^0\ar[r]&
\mathbf{C}_{n}^1\ar[r]&\mathbf{C}_{n}^2\ar[r]&\cdots\ar[r]&\mathbf{C}_n^m\ar[r]&
0}\] is a cotorsion resolution of $\CF$.
\end{proof}

\begin{theorem}\label{109}
If every flat $\CO_X$-module has finite cotorsion dimension. Then
for each $1\leq i\leq m$, every flat $A_i$-module has finite
cotorsion dimension.
\end{theorem}
\begin{proof}
With out lose of generality we can assume that $i=1$.  Let $F$ be a
flat $A_1$-module,   $f:U_1\lrt X$ be the inclusion and
$\xymatrix@C-0.7pc@R-0.9pc{\mathbf{C}_F: 0\ar[r]&F\ar[r]^{\xi}&C^0
\ar[r]^{\delta^0}&C^1\ar[r]^{\delta^1}&C^2\ar[r]^{\delta^2}&\cdots,}$
be  its minimal  cotorsion resolution. By construction,
$\mathbf{C}_F$  is a pure acyclic complex of flat $A_1$-modules.
Apply the exact functor ${f}_{{}_*}$ and get the  pure acyclic
complex
$\xymatrix@C-0.7pc@R-0.9pc{{f}_{{}_*}(\widetilde{\mathbf{C}_F}):
0\ar[r]&{f}_{{}_*}\widetilde{F}
\ar[r]^{{f}_{{}_*}\widetilde{\xi}}&{f}_{{}_*}\widetilde{C^0}
\ar[r]^{{f}_{{}_*}\widetilde{\delta^0}}&{f}_{{}_*}\widetilde{C^1}\ar[r]^{{f}_{{}_*}\widetilde{\delta^1}}
&{f}_{{}_*}\widetilde{C^2}
\ar[r]^{{f}_{{}_*}\widetilde{\delta^2}}&\cdots}$ of flat
$\CO_X$-modules. In fact, it is  a cotorsion resolution of
${f}_{{}_*}\widetilde{F}$.  The assumption  implies that
$\im{f}_{{}_*}\widetilde{\delta^{n-1}}$ is cotorsion for some
integer $n$. So the exact sequence
$$\xymatrix@C-0.7pc@R-0.9pc{0\ar[r]&\im{f}_{{}_*}\widetilde{\delta^{n-1}}\ar[r]&
{f}_{{}_*}(\widetilde{C^n})\ar[r]&\im{f}_{{}_*}\widetilde{\delta^{n}}\ar[r]&0}$$
of flat $\CO_X$-modules  splits. Then
$C^n=\im\delta^{n-1}\oplus\im\delta^{n}$. It follows that $\cd F\leq
n$.

\end{proof}

$\mathbf{Proof~ of ~Theorem ~1.1.}$ $(i)\lrt(ii)$ If every flat
$\CO_X$-module has finite cotorsion dimension. Then by Theorem
\ref{109}, for each $0\leq i \leq m$, every flat $A_i$-module has
finite cotorsion dimension. So,  for each $i$ there exist  an
integer $n_i\geq 0$ such that $A_i$ is $n_i$-perfect. Therefore the
proof of Theorem \ref{nper} implies that $X$ is $n$-perfect for some
integer $n\geq 0$.

$(ii)\lrt(i)$ Clear.

\begin{theorem}\label{perf}
A scheme $X$ is  $n$-perfect if and only if for every $\CO_X$-module
$\CG$, $\cd~\CG\leq n$.
\end{theorem}
\begin{proof}
Let $X$ be $n$-perfect, $\CG$ be an $\CO_X$-module and
$\xymatrix@C-0.7pc@R-0.9pc{0\ar[r]&\CC\ar[r]&\CF'\ar[r]&\CG\ar[r]&0}$
be the flat cover of $\CG$. Then for any flat $\CO_X$-module $\CF$
we have the following exact sequence
\[\xymatrix@C-0.7pc@R-0.9pc{0 = \Ext^{n+1}_X(\CF,\CC)\ar[r]&
\Ext^{n+1}_X(\CF,\CF')\ar[r]&\Ext^{n+1}_X(\CF,\CG)\ar[r]&
\Ext^{n+2}_X(\CF,\CC)=0.}\] Then $\Ext^{n+1}_X(\CF,\CG)=0$ and hence
$\cd~\CG\leq n$.

The converse is trivial.
\end{proof}
Now by using  the main Theorem of  \cite{Si} we give  examples of
non-noetherian $n$-perfect schemes of infinite  Krull dimension.
\begin{example}
Let $R$ be a  ring, $|R|\leq\aleph_n$ for some $n\geq 0$,
$A=R[x_1,x_2,...]$ be the polynomial ring of infinite indeterminate
and $X=\bigcup_{i=1}^{i=m}D(f_i)$ be an open subscheme of Spec$A$.
Then $X$ is a non-noetherian non-affine scheme of infinite krull
dimension(for definitions and notations see \cite[II.2]{H}). By the
same argument that used in the proof of Theorem \ref{nper} we deduce
that $X$ is $k$-perfect for some $k$.
\end{example}
\begin{example}
Let $\mathfrak{T}$ be a topological space of cardinality at most
$\aleph_n$ for some integer  $n\geq 0$. If $\mathfrak{T}$ is not
$p$-space, then the commutative ring $\mathrm{C}(\mathfrak{T})$, the
ring of real valued continuous functions on $\mathfrak{T}$, is  a
non-noetherian $(n+1)$-perfect ring of infinite krull dimension. For
example the metric space $\mathbb{R}$(real numbers) is not a
$p$-space.
\end{example}
\begin{example}
If $R$ is a  noetherian ring of finite krull dimension $n$. Then it
is $n$-perfect.
\end{example}
\begin{example}
If $R$ is $n$-perfect. Then $R[x]$ is also $(n+1)$-perfect.
\end{example}
\begin{example}
The Nagata's example of a noetherian ring of infinite krull
dimension is $n$-perfect for some integer $n$, see[Appendix, Example
1]\cite{Na}
\end{example}
\subsection{Pure injective dimension of flat $\CO_X$-modules}
 Recall that an exact sequence
$\xymatrix@C-0.7pc@R-0.9pc{0\ar[r]&\CK\ar[r]& \CG}$  of
$\CO_X$-modules  is called  pure if it remains exact after tensoring
with any $\CO_X$-module. An $\CO_X$-module $\CE$ is called pure
injective if it is injective with respect pure exact sequences of
$\CO_X$-modules. For a given $\CO_X$-module $\CF$, let  $\CF^*=
\oplus_{i=1}^m{f_i}_*\widetilde{F_i^*}$, $\CF^{**}=
\oplus_{i=1}^m{f_i}_*\widetilde{F_i^{**}}$ such that for each $1\leq
i\leq m$, $F_i=\CF(U_i)$, $F_i^*=\mathrm{Hom}_{\Z}(F_i,{\Q}/{\Z})$,
and $\xymatrix@C-0.7pc@R-0.9pc{f_i:U_i\ar[r]&X}$ be the inclusion.
Then $\CF^*$ and $\CF^{**}$  are pure injective $\CO_X$-modules and
$\CF\lrt\CF^{**}$ is a pure monomorphism.

In this subsection we let $X$ be a \textit{coherent} scheme. Recall
that $X$ is called coherent if $A_i$ is a coherent ring for each
$1\leq i\leq m$

\begin{sproposition}\label{co1}
Let  $\CF$ be a flat $\CO_X$-module. Then $\CF$ is pure injective if
and only if it is cotorsion.
\end{sproposition}
\begin{proof}
Let  $\xymatrix@C-0.7pc@R-0.9pc{
0\ar[r]&\CF\ar[r]&\CC\ar[r]&\CG\ar[r]&0}$ be the cotorsion envelope
of $\CF$. Since $\CF$ is flat then  this sequence is pure and hence
it is split.

Let $\CF$ be a cotorsion $\CO_X$-module and
$\xymatrix@C-0.7pc@R-0.9pc{
0\ar[r]&\CF\ar[r]&\CF^{**}\ar[r]&\frac{\CF^{**}}{\CF}\ar[r]&0}$ be
its pure injective preenvelope. Since $\CF$ and $\CF^{**}$ are flat
$\CO_X$-module  then  $\frac{\CF^{**}}{\CF}$ is also flat and so
this sequence is split.
\end{proof}
The pure injective dimension of an $\CO_X$-module $\CF$
 can be defined in usual sense.
\begin{stheorem}
A scheme $X$ is $n$-perfect if and only if every $\CO_X$-module has
finite pure injective dimension.
\end{stheorem}
\begin{proof}
Let $\CG$ be an $\CO_X$-module and
$$\xymatrix@C-0.7pc@R-0.9pc{0\ar[r]&\CG\ar[r]&\CC^0\ar[r]^{\delta^0}&\CC^1\ar[r]^{\delta^1}&\cdots}$$
be its minimal cotorsion resolution. By Theorem \ref{perf},
$\im\delta^{n-1}$ is cotorsion flat and  by  Proposition \ref{co1}, it is pure
injective. Therefore this pure exact sequence is a pure injective
resolution of $\CG$ of length   $n$.

By Proposition \ref{co1}, the converse is
trivial.
\end{proof}

\section{Application}
Let $\KFX$ be the homotopy category of complexes of flat
$\CO_X$-modules, $\KPFX$ be the full subcategory of $\KFX$
consisting of all pure acyclic complexes of flat $\CO_X$-modules and
$\KCOFX$ be the essential  image of the homotopy category of
complexes of cotorsion flat $\CO_X$-modules.  In \cite{HS}, the authors proved that
there is an equivalence $\KCOFX\lrt\KFX/{\KPFX}=\DFX$ of homotopy
categories, whenever $\KPFX\cap\KCOFX=0$ and $\mathfrak{Qco}X$ have
enough flats. For instance such equivalenece of homotopy categories
exists, when $X$ is $n$-perfect (possibly  non-noetherian of  infinite
Krull dimension).

In this section we let $X$ be a coherent, $\CPFX$ be the category of
all flat complexes of $\CO_X$-modules and
$\mathbf{C}(\mathrm{Pinj}X)$ be the category of complexes of pure
injective  $\CO_X$-modules.

\begin{theorem}\label{i2}
Let $\C$ be a complex of $\CO_X$-modules. Then $\C\in\CPFX^\perp$ if
and only if it is a complex of pure injective $\CO_X$-modules .
\end{theorem}
\begin{proof}
Let $\C$ be a complex of pure injective $\CO_X$-modules. By \cite[Proposition 2.6]{HS}, there is a degree-wise split exact sequence
$\xymatrix@C-0.7pc@R-0.9pc{0\ar[r]&\C\ar[r]&\C'\ar[r]&\F'\ar[r]&0} $
where $\C'\in\CPFX^\perp$  and $\F'\in\CPFX$. Therefore  we have  a
canonical morphism $ u: \mathbf{F}' \rightarrow \Sigma \mathbf{C}$
such that $\xymatrix@C-0.7pc@R-0.9pc{
\mathbf{C}\ar[r]&\mathbf{C}'\ar[r]&\mathbf{F}'\ar[r]^u&\Sigma
\mathbf{C}}$is a triangle in $\KFX$.  Moreover, $\F'$ is a complex
of cotorsion $\CO_X$-modules and   hence it is contractible by
$n$-perfectness of $X$. It follows that  for each flat complex $\F$,
$\Hom_{\K(X)}(\F,\C)\cong\Hom_{\K(X)}(\F,\C')=0$.  Therefore  by
\cite[Proposition 2.1]{HS}, $\C\in\CPFX^\perp$.

The converse follows from \cite[Proposition 2.1]{HS}.
\end{proof}
\begin{corollary}
The cotorsion theory $(\CPFX, \mathbf{C}(\mathrm{Pinj}X))$ is
complete.
\end{corollary}
\begin{proof}
The result follows from \cite[Theorem 2.5]{HS} and Theorem \ref{i2}.
\end{proof}
Let $\KPIFX$ be the essential image(in the sense of \cite{HS}) of
the homotopy category of complexes of pure injective flat
$\CO_X$-modules in $\KFX$.
\begin{corollary}\label{orto}
There is an equivalence $\KPIFX\lrt \DFX$ of homotopy categories.
\end{corollary}
\begin{proof}
The pair $(\KPFX,\KPIFX)$ is a complete cotorsion theory in $\KFX$
in the sense of \cite{HS}. Then there is an equivalence $\KPIFX\lrt
\DFX$ of homotopy categories.
\end{proof}

 %\textbf{Aknowledgements.} The authors are deeply
%grateful to the referee for his/her careful reading of the
%manuscript. We would like to thank the Center of Excellence for
%Mathematics(University of Isfahan).

\end{document}